\documentclass[oneside]{amsart}

\usepackage{amssymb}

\usepackage{graphicx}
\usepackage{hyperref}
\usepackage{booktabs}
\usepackage{algorithm}
\usepackage{algpseudocode} 
\usepackage{amsmath}
\usepackage{mathtools}
\usepackage{float}
\usepackage{geometry}

\newcommand{\Tp}{\mathcal{T}_{\mathfrak{p}}}
\newcommand{\fp}{\mathfrak{p}}

\newtheorem{theorem}{Theorem}[section]
\newtheorem{lemma}[theorem]{Lemma}
\newtheorem{corollary}[theorem]{Corollary}

\theoremstyle{definition}
\newtheorem{definition}[theorem]{Definition}

\theoremstyle{remark}

\numberwithin{equation}{section}

\DeclareMathOperator{\val}{val}
\newcommand{\cO}{\mathcal{O}}

\begin{document}

\title[Fundamental domains]{Fundamental domains for quaternionic $S$-arithmetic groups over totally real fields}

\author[Masdeu]{Marc Masdeu}
\address{}
\curraddr{}
\email{Marc.Masdeu@uab.cat}
\thanks{}

\author[Torrents]{Eloi Torrents}
\address{}
\curraddr{}
\email{Eloi.Torrents@uab.cat}
\thanks{}


\date{\today}
\dedicatory{}

\begin{abstract}
Let $B$ be a totally-definite quaternion algebra over a totally real field $F$, let $\mathfrak{p}$ be a prime ideal of $F$, and let $\Gamma$ be the group of reduced norm-$1$ elements of an Eichler $\mathcal{O}_F[1/\mathfrak{p}]$-order $R$ inside $B$. We give an algorithm to compute the fundamental domain for the action of $\Gamma$ on the Bruhat-Tits tree of $\operatorname{GL}_2(F_\mathfrak{p})$. Using this, we tabulate Shimura curves of genus up to $3$ over any totally real field which can be $\mathfrak{p}$-adically uniformized for some prime $\mathfrak{p}$.
\end{abstract}

\maketitle

\section{Introduction}
\label{sec:introduction}

Let $F$ be a totally real number field, fix a prime 
$\mathfrak{p} \subseteq F$, and let $F_{\mathfrak{p}}$ denote the completion of $F$ at $\mathfrak{p}$.
In this work we develop an algorithm to compute fundamental domains
for the action of certain discrete subgroups of 
$\mathrm{SL}_2(F_{\mathfrak{p}})$ on the Bruhat-Tits tree
associated with $\mathrm{GL}_2(F_{\mathfrak{p}})$.
The discrete groups we consider arise from Eichler orders in definite quaternion algebras defined over~$F$.

For Shimura curves with bad reduction at~$\mathfrak{p}$,
the structure of the bad special fiber is encoded by these fundamental domains.
We have computed\footnote{A \texttt{SageMath} implementation of this algorithms is available at \url{https://github.com/eloitor/btquotients}.}
an extensive collection of examples of fundamental domains
arising from $\mathfrak{p}$-adic uniformizations of Shimura curves.

The rest of this note is organized as follows. In Section~\ref{sec:notation} we introduce the basic notation used throughout the article. Section~\ref{sec:fundoms} contains the description of the algorithms used to compute the fundamental domains.
In Section~\ref{sec:examples} we illustrate these algorithms with some examples. Finally, in Section~\ref{sec:applications} we show how to use our algorithms to tabulate some $\mathfrak{p}$-adic uniformizable Shimura curves of genus up to $3$.

\section{Notation and setup}
\label{sec:notation}

In this section, we introduce certain $S$-arithmetic quaternionic groups acting on a corresponding
Bruhat-Tits tree. Throughout this section, $F$ will denote a totally real number field with ring of integers $\cO_F$.
We will also consider a fixed prime ideal $\fp$ in $\cO_F$. By $F_\fp$ we will denote the completion of $F$ at $\fp$.

\subsection{Quaternion algebras and orders}
For precise definition and basic facts on quaternion algebras we refer the reader to~\cite{voight-quatalgs}.
A \emph{quaternion algebra} over $F$ is a $4$-dimensional $F$-algebra
\[
B = \left( \frac{a,b}{F} \right)
  = F \langle i, j \mid i^2 = a,\; j^2 = b,\; ij = -ji \rangle,
\]
for some $a,b \in F^{\times}$. The ramification set $\operatorname{Ram}(B)$ of $B$ is the set of places $v$
of $F$ for which $B\otimes_F F_v$ is a division algebra. The set $\operatorname{Ram}(B)$ has even cardinality,
and the discriminant of $B$ is the product of those places in $\operatorname{Ram}(B)$ which are finite. We say
that a $B$ is \emph{ramified} at $v$ if $v \in \operatorname{Ram}(B)$, and it is \emph{split} otherwise. Moreover,
$B$ is \emph{definite} if it is ramified at all real places of $F$, and \emph{indefinite} otherwise.

Consider a definite quaternion algebra $B$ over $F$ that splits at $\mathfrak{p}$.
Let $R^{\max} \subseteq B$ be a maximal order, and $R \subseteq R^{\max}$ be an Eichler order of level coprime to $\mathfrak{p}$.
Let $\iota$ be a splitting
\[
  B \otimes_F F_{\mathfrak{p}} \;\;\tilde{\longrightarrow}^{\;\iota\;}\;\; M_2(F_{\mathfrak{p}})
\]
such that $\iota(R^{\max}_{\mathfrak{p}}) = M_2(\mathcal{O}_{F_{\mathfrak{p}}})$.

Let $S = \{\mathfrak{p}\} \cup S_\infty$.
Let $R \subseteq B$ be an Eichler order, and let $R[1/\mathfrak{p}]_1^{\times}$ denote the subgroup of elements of reduced norm $1$ in the order $R[1/\mathfrak{p}]$.
We define the $S$-arithmetic group
\[
  \Gamma = \iota\!\big(R[1/\mathfrak{p}]_1^{\times}\big)
  \;\subseteq\; \mathrm{SL}_2(F_{\mathfrak{p}}).
\]
We will explain an algorithm to compute a fundamental domain for the action of $\Gamma$ on
$\mathcal{T}_{\mathfrak{p}}$, the Bruhat-Tits tree for $\mathrm{GL}_2(F_{\mathfrak{p}})$.



\subsection{The Bruhat-Tits tree}
Let $\mathcal{O}_F$ be the ring of integers of a number field $F$.
We fix a prime ideal $\mathfrak{p} \subseteq \mathcal{O}_F$.
Let $F_{\mathfrak{p}}$ be the completion of $F$ at $\mathfrak{p}$,
and let $\mathcal{O}_{F_{\mathfrak{p}}}$ be its valuation ring, with $\mathfrak{p}$-adic valuation $v_{\mathfrak{p}}$.
Let $\pi \in \mathcal{O}_F$ be a uniformizer for $\mathfrak{p}$ in $\mathcal{O}_{F_{\mathfrak{p}}}$. The following definitions and basic properties can be found in~
\cite{serre2002trees}, we recall them in order to fix notation.

\begin{definition}
The \emph{Bruhat-Tits tree} for $\mathrm{GL}_2(F_{\mathfrak{p}})$
is the graph $\mathcal{T}_{\mathfrak{p}}$ whose vertices are equivalence classes of 
$\mathcal{O}_{F_{\mathfrak{p}}}$-lattices in the two-dimensional vector space $F_{\mathfrak{p}}^2$, modulo homothety.
Two vertices $v_1, v_2$ of $\mathcal{T}_{\mathfrak{p}}$ are connected by a directed edge
if there exist representative lattices $\Lambda_1, \Lambda_2$ belonging to the classes
$v_1, v_2$ such that
\[
  \pi \Lambda_1 \subsetneq \Lambda_2 \subsetneq \Lambda_1.
\]
\end{definition}

The graph $\mathcal{T}_{\mathfrak{p}}$ is a homogeneous tree, and the degree of each vertex is $N(\mathfrak{p}) + 1$.

We denote the homothety class of a lattice $\Lambda \subseteq F_{\mathfrak{p}}^2$ by $[\Lambda]$,
which corresponds to a vertex of the Bruhat-Tits tree.
Each vertex can be represented by a 
$2\times 2$ matrix over $F_{\mathfrak{p}}$ whose column vectors span the corresponding lattice.
Let $v_0$ be the vertex associated with the standard lattice $\mathcal{O}_{F_{\mathfrak{p}}}^2$;
in particular, $v_0$ is represented by the identity matrix.

We consider the left action of $\mathrm{GL}_2(F_{\mathfrak{p}})$ on the vertices of 
$\mathcal{T}_{\mathfrak{p}}$
given by linear transformations. 
Since this action preserves adjacency in the graph, it induces an action on the edges
and hence on the entire tree $T_{\mathfrak{p}}$.

This action is transitive both on vertices and on edges.
The stabilizer of the vertex $v_0$ is 
$F_{\mathfrak{p}}^{\times}\mathrm{GL}_2(\mathcal{O}_{F_{\mathfrak{p}}})$,
so that the set of vertices of the tree can be identified with the quotient
\[
  \mathrm{GL}_2(F_{\mathfrak{p}}) 
  \big/ 
  F_{\mathfrak{p}}^{\times}\mathrm{GL}_2(\mathcal{O}_{F_{\mathfrak{p}}}).
\]

On the other hand, the stabilizer of the directed edge connecting 
$v_0$ to 
$v_1 = [\pi \mathcal{O}_{F_{\mathfrak{p}}} \times \mathcal{O}_{F_{\mathfrak{p}}}]$ 
is
\[
  F_{\mathfrak{p}}^{\times} \mathrm{GL}_2(\mathcal{O}_{F_{\mathfrak{p}}})
  \cap 
  \begin{pmatrix} \pi & 0 \\ 0 & 1 \end{pmatrix}
  F_{\mathfrak{p}}^{\times} \mathrm{GL}_2(\mathcal{O}_{F_{\mathfrak{p}}})
  \begin{pmatrix} \pi & 0 \\ 0 & 1 \end{pmatrix}^{-1}.
\]
Hence, the set of edges of $T_{\mathfrak{p}}$ can be identified with the quotient
\[
  \mathrm{GL}_2(F_{\mathfrak{p}})
  \big/
  F_{\mathfrak{p}}^{\times}\Gamma_0(\mathfrak{p}\mathcal{O}_{F_{\mathfrak{p}}}),
\]
where
\[
  \Gamma_0(\mathfrak{p}\mathcal{O}_{F_{\mathfrak{p}}})
  =
  \left\{
  \begin{pmatrix} a & b \\ c & d \end{pmatrix}
  \in \mathrm{GL}_2(\mathcal{O}_{F_{\mathfrak{p}}})
  \;\middle|\;
  c \in \mathfrak{p}\mathcal{O}_{F_{\mathfrak{p}}}
  \right\}.
\]

We now show that vertices and edges of the Bruhat-Tits tree can be represented
by $2\times 2$ matrices with coefficients in $\mathcal{O}_F$,
which simplifies computational manipulation.
If $S \subseteq \mathcal{O}_F$ is a set of representatives of the quotient
$\mathcal{O}_{F_{\mathfrak{p}}}/\mathfrak{p}\mathcal{O}_{F_{\mathfrak{p}}}$,
then for $n\ge1$ we define
\[
  S_n = S + \pi S + \cdots + \pi^{n-1} S
  \subseteq \mathcal{O}_F,
\]
which forms a set of representatives for
$\mathcal{O}_{F_{\mathfrak{p}}}/\mathfrak{p}^n \mathcal{O}_{F_{\mathfrak{p}}}$.

\begin{lemma}\label{matrix_coeffs_rep_vertices}
The quotients
\[
  \mathrm{GL}_2(F_{\mathfrak{p}})
  \big/
  F_{\mathfrak{p}}^{\times}\Gamma_0(\mathfrak{p}\mathcal{O}_{F_{\mathfrak{p}}})
  \quad\text{and}\quad
  \mathrm{GL}_2(F_{\mathfrak{p}})
  \big/
  F_{\mathfrak{p}}^{\times}\mathrm{GL}_2(\mathcal{O}_{F_{\mathfrak{p}}})
\]
admit systems of representatives consisting of matrices with entries in $\mathcal{O}_F$.
Moreover, the representatives for the edges of the Bruhat-Tits tree can be chosen to have the form
\[
  \begin{pmatrix} \pi^m & 0 \\ r & \pi^n \end{pmatrix}
  \quad\text{with } r \in S_{n+1},
  \qquad\text{or}\qquad
  \begin{pmatrix} 0 & \pi^m \\ \pi^n & r \end{pmatrix}
  \quad\text{with } r \in S_n,
\]
while representatives for the vertices can be chosen as
\[
  \begin{pmatrix} \pi^m & 0 \\ r & \pi^n \end{pmatrix}
  \quad\text{or}\quad
  \begin{pmatrix} 0 & \pi^m \\ \pi^n & r \end{pmatrix},
  \qquad r \in S_n.
\]
\end{lemma}

\begin{proof}
Given a matrix in $\mathrm{GL}_2(F_{\mathfrak{p}})$,
we first scale it by an element of $F_{\mathfrak{p}}^{\times}$ to obtain
\[
  \begin{pmatrix} a & b \\ c & d \end{pmatrix}
  \in \mathrm{GL}_2(\mathcal{O}_{F_{\mathfrak{p}}}),
\]
with one entry having $\mathfrak{p}$-adic valuation~$0$.

Suppose that $v_{\mathfrak{p}}(a) \le v_{\mathfrak{p}}(b)$.
Then the matrix is right-equivalent under 
$\Gamma_0(\mathfrak{p}\mathcal{O}_{F_{\mathfrak{p}}})$
to one of the form
\[
  \begin{pmatrix} a & 0 \\ c & d' \end{pmatrix},
\]
since
\[
  \begin{pmatrix} a & b \\ c & d \end{pmatrix}
  \begin{pmatrix} 1 & -b/a \\ 0 & 1 \end{pmatrix}
  =
  \begin{pmatrix} a & 0 \\ c & d' \end{pmatrix}.
\]
Writing $a = \alpha \pi^n$ and $d' = \delta \pi^m$, we may scale by
$\begin{psmallmatrix}\alpha & 0 \\ 0 & \delta\end{psmallmatrix}$.
Finally, $c$ can be adjusted to lie in $S_{m+1}$ by acting with
\[
  \begin{pmatrix}
  1 & 0 \\
  \frac{(c\pmod{\pi^{m+1}}) - c}{\pi^m} & 1
  \end{pmatrix}
  \in \Gamma_0(\mathfrak{p}\mathcal{O}_{F_{\mathfrak{p}}}).
\]

If instead we have $v_{\mathfrak{p}}(a) > v_{\mathfrak{p}}(b)$,
then $a/b \in \mathfrak{p}\mathcal{O}_{F_{\mathfrak{p}}}$, and
\[
  \begin{pmatrix} a & b \\ c & d \end{pmatrix}
  \begin{pmatrix} 1 & 0 \\ -a/b & 1 \end{pmatrix}
  =
  \begin{pmatrix} 0 & b \\ c' & d \end{pmatrix}.
\]
As before, writing $c' = \gamma \pi^n$ and $b = \beta \pi^m$,
we can scale to make these powers of $\pi$,
and $d$ can be adjusted to lie in $S_n$ by acting with
\[
  \begin{pmatrix}
  1 & \frac{(d\pmod{\pi^n}) - d}{\pi^n} \\ 0 & 1
  \end{pmatrix}
  \in \Gamma_0(\mathfrak{p}\mathcal{O}_{F_{\mathfrak{p}}}).
  \qedhere
\]
\end{proof}

\section{Fundamental domains of the Bruhat-Tits tree}
\label{sec:fundoms}

We adapt the algorithm described in \cite{francmasdeu2014} to compute a fundamental domain of $\Tp$ for the group $\Gamma = \Gamma^\mathfrak{p}_{\mathfrak{N}^-, \mathfrak{N}^+}$.
The core of the algorithm involves a procedure for checking whether two given edges or vertices are equivalent under the group action and, if so, providing an element of the group realizing the equivalence.

More concretely, given two matrices $u$, $v$ representing two vertices (or two edges), we are interested in deciding if they are $\Gamma$-equivalent, by obtaining an element $g \in R[1/\mathfrak{p}]_1^{\times}$ such that $\iota(g)u = v$, if it exists.

We define the distance between two vertices of $\mathcal{T}_{\mathfrak{p}}$ as the length of the path connecting them.

\begin{lemma}
If two vertices or two edges are $\Gamma$-equivalent, the distance between them must be even.
\end{lemma}

\begin{proof}
The Corollary to Proposition~1 of \cite[Chapter~2, Subsection~1.2]{serre2002trees}
applies in this setting, since $\Gamma \subseteq \mathrm{GL}_2(F_{\mathfrak{p}})$.
\end{proof}

We represent the two vertices/edges $u,v$ by reduced matrices in $M_2(\mathcal{O}_F)$ as in Lemma~\ref{matrix_coeffs_rep_vertices}, 
with $\det(u) = \pi^a$ and $\det(v) = \pi^b$.
We write $\mathrm{Hom}_\Gamma(u,v)$ for the set of elements of $\Gamma$ that send $u$ to $v$.

Let $G$ be the group $\mathrm{GL}_2(\mathcal{O}_{F,\mathfrak{p}})$ when working with vertices, and $\Gamma_0(\mathfrak{p}\mathcal{O}_{F,\mathfrak{p}})$ when working with edges. We have
\[
\begin{aligned}
  \mathrm{Hom}_\Gamma(u,v)
  &= \Gamma \cap \{ g \in \mathrm{GL}_2(F_\mathfrak{p}) \mid g \cdot [u] = [v] \} \\
  &= \Gamma \cap \{ g \in \mathrm{GL}_2(F_\mathfrak{p}) \mid g \cdot u \, F_\mathfrak{p}^\times G = v \, F_\mathfrak{p}^\times G \} \\
  &= \Gamma \cap \{ v^{-1} \, g \, u \mid g \in F_\mathfrak{p}^\times G \} \\
  &= \Gamma \cap v^{-1} F_\mathfrak{p}^\times G \, u.
\end{aligned}
\]

Let $\Lambda_0$ be the lattice $M_2(\mathcal{O}_{F,\mathfrak{p}})$ when working with vertices, and
\[
  \Lambda_0 = M_0(\mathfrak{p}\mathcal{O}_{F,\mathfrak{p}}) 
=
  \left\{
    \begin{pmatrix}
      a & b \\ c & d
    \end{pmatrix}
    \in M_2(\mathcal{O}_{F,\mathfrak{p}})
    \;\middle|\;
    c \in \mathfrak{p}\mathcal{O}_{F,\mathfrak{p}}
    \right\}.
\]

\begin{lemma}\label{GammaHom}
If the distance between two vertices or two edges of $\mathcal{T}_{\mathfrak{p}}$ represented by the pair of matrices $u,v \in M_2(\mathcal{O}_{F,\mathfrak{p}})$ is odd, then $\mathrm{Hom}_\Gamma(u,v) = \emptyset$.  
Otherwise, let $2m = a + b$, where $a=\val_\pi(\det u)$ and $b=\val_\pi(\det v)$ as before. Then
\[
  \mathrm{Hom}_\Gamma(u,v)
  = \Gamma \cap \pi^{-m} v^{*} \Lambda_0 u.
\]
\end{lemma}

\begin{proof}
We check that the proof from \cite[ Lemma 3.1]{francmasdeu2014} generalizes in our setting.
We have to prove that 
\[
\Gamma \cap (v^{-1} F_\mathfrak{p}^\times G u) 
= \Gamma \cap \pi^{-m} v^* \Lambda_0 u.
\]

Let $z = v^{-1} \lambda g u \in \Gamma \cap (v^{-1} F_\mathfrak{p}^\times G u)$, 
with $\lambda \in F_\mathfrak{p}^\times$ and $g \in G$.
Taking the determinant,
\[
\det(z) = \pi^{b-a} \lambda^2 \det(g) = \pi^{m - 2a} \lambda^2 \det(g).
\]
Since $\det(z)$ and $\det(g)$ are in $\mathcal{O}_{F, \mathfrak{p}}^\times$,
we have $\val_\mathfrak{p} (\pi^{-a}\lambda) = -m$, and consequently
$\pi^{-a}\lambda g \in \pi^{-m} M_2(\mathcal{O}_{F, \mathfrak{p}})$. Therefore,
\[
\Gamma \cap (v^{-1} F_\mathfrak{p}^\times G u)
\subseteq \Gamma \cap \pi^{-m} v^* \Lambda_0 u.
\]

Conversely, given $z \in \Gamma \cap \pi^{-m} v^* \Lambda_0 u \subseteq SL_2(F_\mathfrak{p})$,
we have $z = \pi^{-m} v^* g u$ for some $g \in \Lambda_0$.
Taking the determinant, $\det z = \pi^{a+b-2m} \det(g)$, and therefore $g \in G$.
\end{proof}

Since $B$ is definite, we can consider the following filtration of $\Gamma$ by finite sets:
\[
\Gamma = \bigcup_{t \ge 1} \Gamma_t, 
\qquad 
\Gamma_t = \left\{ \iota\left(\frac{x}{\delta^t}\right) \;\middle|\; x \in R, \ \mathrm{nrd}(x) = \delta^{2t} \right\}.
\]

\begin{corollary}\label{CorollaryCeil}
If a pair of vertices or edges $u$ and $v$ are $\Gamma$-equivalent, there must exist a quaternion
\[
q = \frac{x}{\delta^{\lceil m/d \rceil}} \in R[1/\mathfrak{p}],
\]
where $m$ is defined as in Lemma~\ref{GammaHom}, $x \in R$ and $\mathrm{nrd}(x) = \delta^{2\lceil m/d \rceil}$, 
such that $\iota(q) \in \mathrm{Hom}_\Gamma (u,v) \cap \Gamma_{\lceil m/d \rceil}$.
\end{corollary}

\begin{proof}
Let $g \in \mathrm{Hom}_\Gamma (u,v)$, and write $g = \pi^{-m} v^* x u = \iota(q)$ as in Lemma~\ref{GammaHom},
where $x \in \Lambda_0$ and $q = k/\delta^n \in R[1/\mathfrak{p}]_1^\times$ for some $k \in R$ and $n \in \mathbb{N}$.
We prove that $\delta^{\lceil m/d \rceil} q \in R$ and thus $\iota(q) \in \Gamma_{\lceil m/d \rceil}$.

Consider $\lambda' = \iota^{-1}(\pi^m g)$, with reduced norm $\pi^{2m}$.
We have $\lambda' \in R$, since on one hand $\lambda' = \pi^m q \in R[1/\mathfrak{p}]$,
and on the other $\lambda' = \iota^{-1}(v^* x u) \in R^\mathrm{max}_\mathfrak{p}$.
Therefore $\lambda' \in R[1/\mathfrak{p}] \cap R^\mathrm{max}_\mathfrak{p} = R$.

Considering $\tau = \pi^d / \delta \in \mathcal{O}_F$, and let $h = d\lceil m/d \rceil - m$.
Then
\[
\delta^{\lceil m/d \rceil} q 
= \frac{\pi^{d\lceil m/d \rceil}}{\tau^{\lceil m/d \rceil}} q 
= \frac{\pi^h \pi^m}{\tau^{\lceil m/d \rceil}} q 
= \frac{\pi^h \lambda'}{\tau^{\lceil m/d \rceil}}.
\]
As before, $\delta^{\lceil m/d \rceil} q \in R[1/\mathfrak{p}]$, 
and we have $\delta^{\lceil m/d \rceil} q \in R^\mathrm{max}_\mathfrak{p}$ since $\tau \in \mathcal{O}_{F,\mathfrak{p}}^\times$.
\end{proof}

\medskip
A simple approach to computing such $q$, inspired by a comment found in M.Greenberg's Ph.D. thesis \cite{greenberg2006heegner},
consists in enumerating all quaternions $q \in R$ having reduced norm $\delta^{2\lceil m/d \rceil}$,
and checking if any of them satisfies $\iota(q)\cdot u = v$.

The enumeration could be done using that 
$\mathrm{Tr}_{F/\mathbb{Q}} \circ \mathrm{nrd}: B \to \mathbb{Q}$
is a definite quadratic form. Alternatively, we
 can use the LLL-algorithm to enumerate all quaternions $q \in R$ with 
\[
\mathrm{Tr}_{F/K}(\mathrm{nrd}(q)) = \mathrm{Tr}_{F/K}\bigl(\delta^{\lceil 2m/\delta \rceil}\bigr),
\]
and filter the ones that satisfy $\mathrm{nrd}(q) = \delta^{\lceil 2m/d \rceil}$.

We now provide a more efficient method influenced by the work of \cite{francmasdeu2014}.

Let $u$ and $v$ be two matrices representing two vertices or edges of the Bruhat-Tits tree, and assume that $u$ and $v$ are written in reduced form as in \ref{matrix_coeffs_rep_vertices}.
Let $2m = \val_{\mathfrak{p}}\!\big(\det(v\,u)\big)$, and let $h = d\,\lceil m/d\rceil - m$.

The problem of determining whether two vertices or edges of the Bruhat-Tits tree are $\Gamma$-equivalent, and finding an element $\gamma \in \Gamma$ that realizes the equivalence, can be reduced to finding an element $\lambda$ of reduced norm $\delta^{\,2\lceil m/d\rceil}$ in the following $\mathcal{O}_F$-lattice of rank $4$:
\[
  \Lambda_{u,v} = \iota^{-1}\!\big(\pi^{h} v^{*}\Lambda_0 u\big)\;\cap\; R \;+\; \mathfrak{p}^{\,d\lceil m/d\rceil + 1} R.
\]

\begin{lemma}\label{areEquivalent_lattice}
If the lattice $\Lambda_{u,v}$ has an element $\lambda$ of reduced norm $\delta^{\,2\lceil m/d\rceil}$, then
\[
  \gamma_\lambda \;=\; \iota\!\left(\frac{\lambda}{\delta^{\lceil m/d\rceil}}\right)
\]
is an element of $\Gamma$ such that $\gamma_\lambda \cdot u = v$.
Otherwise, $u$ and $v$ are not $\Gamma$-equivalent.
\end{lemma}

\begin{proof}
The $\mathfrak{p}$-adic valuation of the reduced norm of the quaternions in the lattice $\Lambda_{u,v}$ is at least
$2m + 2h = 2d\lceil m/d\rceil$, since the determinant of $\pi^{h} v^{*}\Lambda_0 u$ also satisfies this bound, and $\iota$ converts reduced norms into determinants.

From the definition, if $\lambda \in \Lambda_{u,v}$ has reduced norm $\delta^{\,2\lceil m/d\rceil}$, then using Corollary~\ref{CorollaryCeil} we see that $\gamma_\lambda \in \mathrm{Hom}_\Gamma(u,v)$.

On the other hand, assuming that $u$ and $v$ are $\Gamma$-equivalent, by Corollary~\ref{CorollaryCeil} there is an element
$q = x/\delta^{\lceil m/d\rceil} \in R[1/\mathfrak{p}]^{\times}_1$ such that
$\iota\!\big(q\,\delta^{\lceil m/d\rceil}\big) = \pi^{h} v^{*} y u$ for some $y \in \Lambda_0$.
The element $\lambda = q\,\delta^{\lceil m/d\rceil}$ has reduced norm $\delta^{\,2\lceil m/d\rceil}$.
Again, Corollary~\ref{CorollaryCeil} implies that $\lambda$ belongs to $\Lambda_{u,v}$.
\end{proof}

Since $\iota$ maps the reduced norm of a quaternion in $B_{\mathfrak{p}}$ to the determinant of the corresponding matrix in $M_2(F_{\mathfrak{p}})$,
the reduced norm of all quaternions in $\Lambda_{u,v}$ is a multiple of $\delta^{\,2\lceil m/d\rceil}$.
To find an element of reduced norm $\delta^{\,2\lceil m/d\rceil}$ in $\Lambda_{u,v}$, we view $\Lambda_{u,v}$ as a $\mathbb{Z}$-lattice of rank $4n$.

The quadratic form $\operatorname{Tr}_{F/\mathbb{Q}} \circ \operatorname{nrd}$ on $\Lambda_{u,v}$ is positive definite, thus its set of shortest nonzero vectors is finite.
Since isometries preserve the set of shortest vectors,
they coincide with the shortest vectors of the definite quadratic form
\[
  q(x) \;=\; \operatorname{Tr}_{F/\mathbb{Q}}\!\left(\frac{\operatorname{nrd}(x)}{\delta^{\,2\lceil m/d\rceil}}\right).
\]
The value of this quadratic form at an element of reduced norm $\delta^{\,2\lceil m/d\rceil}$ is $\operatorname{Tr}_{F/\mathbb{Q}}(1) = [F:\mathbb{Q}]$.
Thus, we can use the LLL algorithm to list the vectors $\lambda \in \Lambda_{u,v}$ such that $q(\lambda) = \operatorname{Tr}_{F/\mathbb{Q}}(1)$.
If the list is empty, then $\mathrm{Hom}_\Gamma(u,v) = \emptyset$.
Otherwise, we have found an element such that $\iota\!\big(\lambda/\delta^{\lceil m/d\rceil}\big) \in \mathrm{Hom}_\Gamma(u,v)$.

\medskip
A basis for the lattice $\Lambda_{u,v}$ may be computed with precision $d\lceil m/d\rceil + 1$ as follows.

First fix a $\mathbb{Z}$-basis $B$ of the Eichler order $R$, consisting of $4d$ quaternions.
Let $e_{\mathfrak{p}}$ and $f_{\mathfrak{p}}$ be the ramification index and inertial degree of $\mathfrak{p}$ respectively.
The $\mathbb{Z}_p$-module $\mathcal{O}_{F,\mathfrak{p}}$ has rank $e_{\mathfrak{p}} f_{\mathfrak{p}}$, and fix a $\mathbb{Z}_p$-basis for $M_2(\mathcal{O}_{F,\mathfrak{p}})$.

Let $L$ be the $4d \times 4 f_{\mathfrak{p}} e_{\mathfrak{p}}$ matrix representing the embedding
\[
  R \hookrightarrow R^{\max}\otimes_{\mathcal{O}_F} \mathcal{O}_{F,\mathfrak{p}}
  \;\;\xrightarrow{\ \iota\ }\;\; M_2(\mathcal{O}_{F,\mathfrak{p}})
\]
in the chosen bases, with $p$-adic precision $2d\lceil m/d\rceil$.

Let $Z$ be the $(4 e_{\mathfrak{p}} f_{\mathfrak{p}})\times(4 e_{\mathfrak{p}} f_{\mathfrak{p}})$ matrix representing the $\mathcal{O}_{F,\mathfrak{p}}$-lattice $v^{*}\Lambda_0 u$ with the same precision $2d\lceil m/d\rceil$.
Then the lattice $\iota^{-1}\!\big(\pi^{h} v^{*}\Lambda_0 u\big)\cap R$ can be described, with precision $2d\lceil m/d\rceil$, as
\[
  \Bigl\{\, xB \;\Bigm|\; x \in \mathbb{Z}^{4d},\ \exists\, y \in \mathbb{Z}_p^{\,4 e_{\mathfrak{p}} f_{\mathfrak{p}}}
  \text{ such that } xL \equiv yZ \pmod{p^{\,2d\lceil m/d\rceil}} \,\Bigr\}.
\]

Equivalently, this is the $\mathbb{Z}$-lattice spanned by those $x \in \mathbb{Z}^{4d}$ for which there exists
$y \in \mathbb{Z}_p^{\,4e_{\mathfrak p}f_{\mathfrak p}}$ such that
\[
  \begin{bmatrix} x & y \end{bmatrix}
  \begin{bmatrix} L \\[2pt] -Z \end{bmatrix}
  \equiv 0 \pmod{p^{\,2d\lceil m/d\rceil}}.
\]
Finally, add the vectors of $\mathfrak{p}^{\,d\lceil m/d\rceil + 1}R$ (expressed in the chosen $\mathbb{Z}$-basis of $R$) to this lattice, and reduce the basis.

We define two auxiliary functions in Algorithm~\ref{alg:gamma_equivalence} that are used in Algorithm~\ref{alg:fundom} 
to determine equivalences among vertices and edges when computing fundamental domains.

Let \texttt{$\Gamma$-equivalent\_edge} be a function that takes a list of edges and an edge $e$.
If there exists an edge $e'$ in the list and some $\gamma \in \Gamma$ such that
\[
  \iota(\gamma)\, e = e',
\]
it returns the pair $(e', \gamma)$. Otherwise, it returns nothing.

Similarly, let \texttt{$\Gamma$-equivalent\_vertex} be a function that takes a list of vertices
and a vertex $v$, and returns $(v', \gamma)$ if $\iota(\gamma)\, v = v'$ 
for some vertex $v'$ in the list and some $\gamma \in \Gamma$.
Otherwise, it returns nothing.

\begin{algorithm}[H]
\caption{Check $\Gamma$-equivalence of vertices}\label{alg:gamma_equivalence}
\begin{algorithmic}[1]
\Require Prime ideal $\mathfrak{p}$; definite quaternion algebra $B$; Eichler order $R\subseteq B$ as above; splitting $\iota_{\mathfrak{p}}: R_{\mathfrak{p}}\simeq M_2(\mathcal{O}_{F_{\mathfrak{p}}})$; matrices $u,v\in M_2(\mathcal{O}_{F_{\mathfrak{p}}})$ representing vertices of the Bruhat-Tits tree $\mathcal{T}_{\mathfrak{p}}$.
\Ensure A boolean indicating whether the vertices are $\Gamma$-equivalent; if true, also a quaternion $\gamma$ with $\iota(\gamma)[u]=[v]$.

\State $m \gets (v_\mathfrak{p}(\det v_1)+v_\mathfrak{p}(\det v_2)) / 2$
\If{$m\notin \mathbb{Z}$} 
    \State \Return \textbf{False}
\EndIf
\State $h \gets d\,\lceil m/d\rceil - m$ \Comment{Now $v_{\mathfrak{p}}(\det \pi^{h} v_1 v_2)=2d\lceil m/d\rceil$}

\Statex \textit{Precomputations}
\State $L \gets \bigl(\iota(x)\bmod \mathfrak{p}^{\,2\lceil m/d\rceil+1}\bigr)_{x\in \mathrm{Basis}(R)}$ 
\State $K \gets \mathrm{left\text{-}kernel}(L)$
\State $Q_{\mathrm{nrd}} \gets \bigl(\mathrm{trd}(x\,\overline{y})\bigr)_{x,y\in \mathrm{Basis}(R)}$ \Comment{Quadratic form for $\mathrm{nrd}:B\to \mathbb{Q}$}
\State $\Lambda_0 \gets M_2(\mathcal{O}_{F_{\mathfrak{p}}})$
\State $P \gets \bigl(\ell r\bigr)_{\ell\in \mathrm{Basis}(\mathfrak{p}^{\,2d\lceil m/d\rceil+1}),\, r\in \mathrm{Basis}(R)}$

\Statex \textit{Define the lattice $\Lambda$ from Lemma \ref{GammaHom}}
\State $Z \gets \bigl(v^{*}\, b\, u \bmod \mathfrak{p}^{\,2d\lceil m/d\rceil+1}\bigr)_{b\in \mathrm{Basis}(\Lambda_0)}$
\State $S \gets \mathrm{solve\_left}(L,Z)$

\State $M \gets$ first $4n$ columns of $\mathrm{row\_span}\!\left(
\begin{bmatrix}
K & 0\\
S & -I_{4e_{\mathfrak{p}} f_{\mathfrak{p}}}
\end{bmatrix}\right)$

\State $\mathrm{Basis}(\Lambda_{u,v}) \gets \mathrm{row\_reduce}\!\left(\begin{bmatrix} M\\ P\end{bmatrix}\right)$

\Statex \textit{Define a quadratic form on $\Lambda_{u,v}$}
\State $Q \gets \bigl(\mathrm{Tr}_{F/\mathbb{Q}}(Q_{\mathrm{nrd}}[i,j])/\delta^{2m}\bigr)_{1\le i,j\le 4}$
\State $Q_{\Lambda} \gets \mathrm{Basis} (\Lambda_{u,v}) \, \cdot \, Q \, \cdot \, \mathrm{Basis}(\Lambda_{u,v})^{-1}$
\State $\lambda \gets$ shortest vector of $Q_{\Lambda}$

\If{$Q_{\Lambda}(\lambda) > \mathrm{Tr}_{F/\mathbb{Q}}(1)$}
  \State \Return \textbf{False}
\EndIf

\State $q \gets \lambda \cdot \mathrm{Basis}(\Lambda_{u,v})$
\State \Return \textbf{True}, $q$
\end{algorithmic}
\end{algorithm}

\begin{algorithm}[H]
\caption{Compute fundamental domain}\label{alg:fundom}
\begin{algorithmic}[1]
\Require A prime ideal $\mathfrak{p}$, an Eichler order $R \subseteq B$ as above, and a splitting $\iota_{\mathfrak{p}}: R_{\mathfrak{p}} \cong M_2(\mathcal{O}_{F_\mathfrak{p}})$.
\Statex Optionally \texttt{max\_genus} can be set to end the computation early if the genus of the quotient graph exceeds \texttt{max\_genus}.
\Ensure A fundamental domain graph $G = (V, E)$ for the action of $\Gamma$ on $\mathcal{T}_{\mathfrak{p}}$, along with relation arrays \texttt{edges\_relations} and \texttt{vertex\_relations} that encode the $\Gamma$-equivalences between boundary elements.
\Statex Returns \emph{"Limit exceeded"} if the genus exceeds \texttt{max\_genus}.
\vspace{0.5em}
\State $\text{pending\_vertices} \gets [v_0]$
\State $G = (E, V) \gets (\emptyset, \emptyset)$
\State $\text{genus} \gets 0$
\State $\text{edges\_relations} \gets [\;]$
\State $\text{vertex\_relations} \gets [\;]$
\vspace{0.5em}
\While{$\text{pending\_vertices} \neq \emptyset$}
  \State $v \gets \text{pop}(\text{pending\_vertices})$
  \For{$e \in \text{edges\_leaving}(v)$}
    \State $(e', \gamma_e) \gets \Gamma\text{-equivalent\_edge}(E, e)$
    \If{$e' = \emptyset$}
      \State $E \gets \text{append}(E, e)$
      \State $v_t \gets \text{target}(e)$
      \State $(v'_t, \gamma_{v'_t}) \gets \Gamma\text{-equivalent\_vertex}(V, v_t)$
      \If{$v'_t = \emptyset$}
        \State $V \gets \text{append}(V, v_t)$
        \State $\text{pending\_vertices} \gets \text{append}(\text{pending\_vertices}, v_t)$
      \Else
        \State $\text{vertex\_relations}[v_t] \gets \gamma_{v'_t}$
        \State $\text{genus} \gets \text{genus} + 1$
        \If{$\text{genus} > \text{max\_genus}$}
          \State \Return \emph{"Limit exceeded"}
        \EndIf
      \EndIf
    \Else
      \State $\text{edges\_relations}[e] \gets \gamma_e$
    \EndIf
  \EndFor
\EndWhile
\vspace{0.3em}
\State \Return $G, \text{edges\_relations}, \text{vertex\_relations}$
\end{algorithmic}
\end{algorithm}

Once a fundamental domain has been computed and stored as a finite set of vertices and edges,
we can use the algorithm described in~\ref{areEquivalent_lattice}
to reduce any given edge $e$ of the Bruhat-Tits tree to an edge $\tilde{e}$ belonging to the fundamental domain, 
while providing an element $\gamma \in \Gamma$ such that
$\tilde{e} = \gamma e$.

However, the algorithm used to decide the equivalence between pairs of edges involves
finding short vectors in a lattice, which is inefficient.
If one wishes to reduce an edge to the computed fundamental domain by exhaustive search,
the procedure would require as many comparisons as the number of edges in the fundamental domain.

In addition to storing the vertices and edges that define a fundamental domain,
it is useful to precompute and store the reduction of edges from one additional layer of the Bruhat-Tits tree into the fundamental domain.
We refer to this information as \emph{boundary data}.

\begin{lemma}
Let $v$ be a vertex of the Bruhat-Tits tree, and let $\mathrm{dist}(v)$ denote its distance from the base vertex $v_0$.
There exists an algorithm to find a vertex in the fundamental domain that is $\Gamma$-equivalent to $v$,
using at most $(N(\mathfrak{p})+1)\,\mathrm{dist}(v)$ queries to the stored boundary data.
\end{lemma}

\begin{proof}
Once the boundary data are precomputed, we can determine the edge in the fundamental domain equivalent to any given edge~$e$ 
by considering the path
\[
[v_0, e_0, v_1, e_1, \ldots, v_{n-1}, e_n],
\]
where the final edge $e_n$ coincides with~$e$, and each $e_i$ connects the vertices $v_i$ and $v_{i+1}$.

Suppose $e_i$ is the first edge of the path lying outside the fundamental domain.
It belongs to the boundary, and therefore we can recover $\gamma_i$ and $\tilde{e}_i$ in constant time.
The action of $\gamma_i$ on the entire path maps $v_i$ to $\gamma_i v_i$ (which lies inside the fundamental domain) and sends the edge $e_n$ to $\gamma_i e_n$.
The new edge $\gamma_i e_n$ remains $\Gamma$-equivalent to $e$, but its path to the fundamental domain is one edge shorter.

By iterating this process at most $n = \mathrm{dist}(v)$ times, 
we obtain an edge $\tilde{e}$ in the fundamental domain equivalent to~$e$, 
together with the element
\[
\gamma = \gamma_n \gamma_{n-1} \cdots \gamma_i \in \Gamma
\]
realizing this equivalence.
\end{proof}

\section{Examples}
\label{sec:examples}
When the prime $\mathfrak{p}$ is unramified of inertia degree $1$, the local field satisfies
$F_{\mathfrak{p}} \cong \mathbb{Q}_p$.
In this case, the vertices and edges of the Bruhat-Tits tree can be represented by
$2 \times 2$ integer matrices.

Let $F = \mathbb{Q}(\phi)$, where $\phi$ is a root of $x^2 - x - 1$.
We compute the fundamental domain for a maximal order in the definite quaternion algebra over $F$ of discriminant
$\mathfrak{N}^- = 1$,
at the prime $\mathfrak{p} = (-5\phi + 2)$, which has norm $31$.

We measured the execution time of Algorithm~\ref{alg:gamma_equivalence} for a sample of $100$ vertices of the Bruhat-Tits tree associated to $\mathfrak{p}$, 
at different distances from the origin.
All timings were obtained on an Apple M4 processor with 16~GB of RAM. 
The results are summarized in Figure~\ref{fig:is_equivalent_timing}.

\begin{figure}
    \centering
    \includegraphics[width=0.8\linewidth]{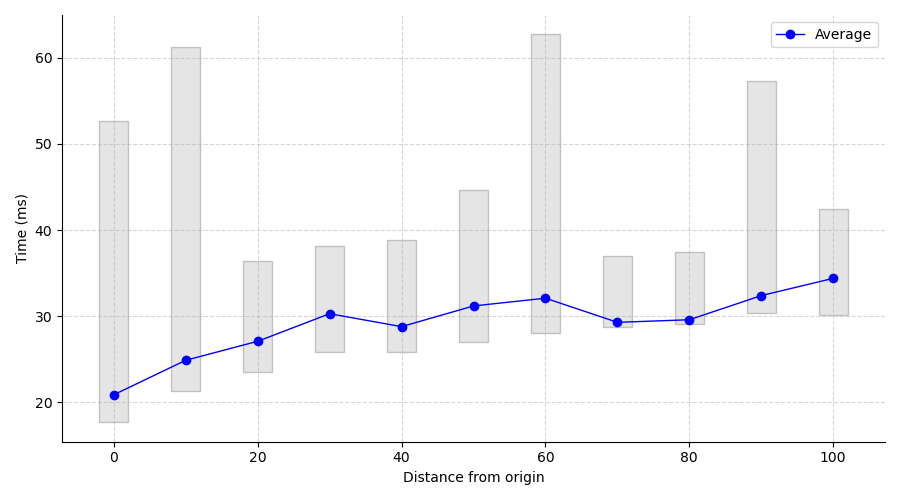}
   \caption{Running time of Algorithm~\ref{alg:gamma_equivalence} for a sample of $100$ random vertices at different distances from the origin. Gray boxes indicate maximum and minimum times.} 
    \label{fig:is_equivalent_timing}
\end{figure}

We observe that the running time grows slowly when comparing distant vertices or edges.
This suggests that our algorithm efficiently computes fundamental domains even for trees of considerable diameter.

In the next examples, we illustrate the output of Algorithm~\ref{alg:fundom}.
In the quadratic field $\mathbb{Q}(\sqrt{97})$, the prime $3$ splits as
\[
  (3) = (10 - \sqrt{97})(10 + \sqrt{97}).
\]
For $\mathfrak{p} = (10 - \sqrt{97})$ and $\mathfrak{N}^- = \mathfrak{N}^+ = 1$,
the corresponding fundamental domain is shown below.

\begin{figure}[ht!]
\centering
\includegraphics[width=0.4\textwidth]{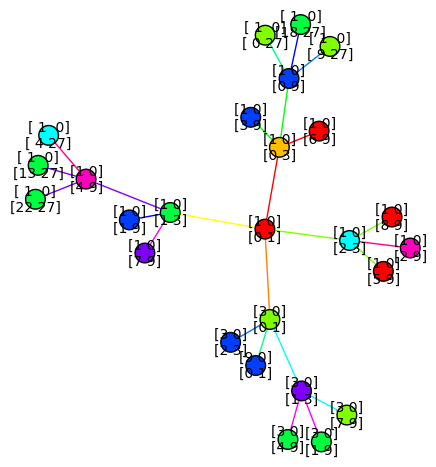}
\hfill
\includegraphics[width=0.4\textwidth]{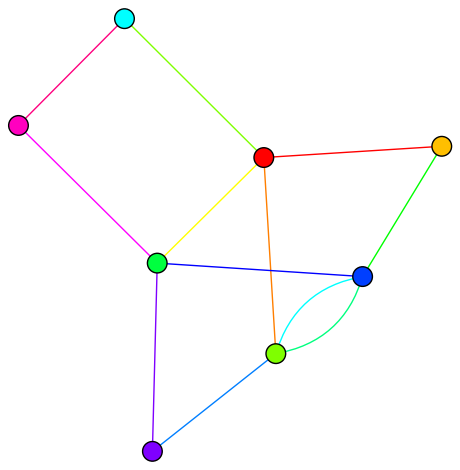}
\caption{
Fundamental domain and boundary data (left), and the quotient of the Bruhat-Tits tree (right).
}
\label{fig:fundamental_domain_quotient}
\end{figure}

As a final example, consider the quadratic field $F = \mathbb{Q}(\sqrt{5})$.
Let $\mathfrak{p} = \left( \tfrac{3}{2}\sqrt{5} - \tfrac{1}{2} \right)$, and take the quaternion algebra over $F$ of discriminant $\mathfrak{N}^- = (21)$.
Let $\mathfrak{N}^+ = (1)$, so that the Eichler order is maximal.
For these data, the computed fundamental domain has $16$ vertices and $80$ edges, corresponding to a Shimura curve of genus $65$.
This computation required approximately $30$ seconds on the same hardware as above.

\begin{figure}[ht!]
\centering
\includegraphics[width=0.8\textwidth]{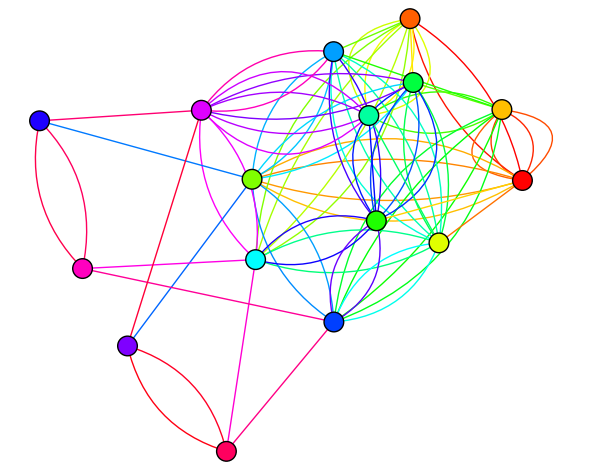}
\caption{
Fundamental domain of genus~$65$, with $16$ vertices and $80$ edges.
}
\label{fig:example_big_graph}
\end{figure}

\section{Applications}
\label{sec:applications}
The fundamental domains computed above have several applications. For example, in~\cite{francmasdeu2014} the authors use them to compute equations for certain Shimura curves. In this section, we illustrate how these fundamental domains yield $\mathfrak{p}$-adic uniformizations of Shimura curves.

Let $B$ be a quaternion algebra over a totally real field $F$, ramified at all but one infinite place, and let $\mathfrak{p}$ be a finite prime dividing its discriminant.
Write the discriminant of $B$ as $\mathfrak{p}\mathfrak{N}^-$, and let $\mathfrak{N}^+$ be an ideal coprime to $\mathfrak{p}\mathfrak{N}^-$.

The Čerednik-Drinfel'd theorem provides a $\mathfrak{p}$-adic uniformization of the Shimura curve 
$X_{\mathfrak{p}\mathfrak{N}^-, \mathfrak{N}^+}$ by a rigid-analytic curve 
$X^{\mathfrak{p}}_{\mathfrak{N}^-, \mathfrak{N}^+}$.

The special fiber $C$ of the Drinfel'd integral model of the Shimura curve is a semistable curve whose irreducible components are isomorphic to $\mathbb{P}^1$.
Their incidence relations can be represented by a reduction graph $G$, which has a vertex for each component and an edge for each node.
This graph is canonically identified with the quotient
$\Gamma^{\mathfrak{p}}_{\mathfrak{N}^-, \mathfrak{N}^+} \backslash \mathcal{T}_{\mathfrak{p}}$
\cite[Corollary~3.1.16]{milione2015}.

The arithmetic genus of this special fiber can be computed using the formula
\cite[Lemma~3.18, §10.3]{Liu2002}
\[
  p_a(C) = \beta(G) + \sum_{1 \le i \le n} p_a(\Gamma'_i),
\]
where $\beta(G)$ is the first Betti number of $G$.
Since the components $\Gamma'_i$ are isomorphic to $\mathbb{P}^1$, their genus is zero, and thus $p_a(C) = \beta(G)$.

For all Shimura curves defined over totally real number fields of degrees between $2$ and $7$, whose genus is at most $3$, we have verified whether a 
$\mathfrak{p}$-adic uniformization satisfying our requirements is available.
In particular, the prime $\mathfrak{p}$ must be unramified of inertia degree~$1$.

For those cases where a $\mathfrak{p}$-adic uniformization is available, we chose an Eichler order of level $\mathfrak{N}^+$ and computed the corresponding fundamental domain.
We recorded all such fundamental domains with genus at most~$3$.

In order to reduce redundant calculations, we consider two uniformizations given by the data 
$(\mathfrak{p}_1, \mathfrak{N}^-_1, \mathfrak{N}^+_1)$ and $(\mathfrak{p}_2, \mathfrak{N}^-_2, \mathfrak{N}^+_2)$ 
to be equivalent if there exists an automorphism $\sigma$ of $F$ such that 
$\sigma(\mathfrak{p}_1) = \mathfrak{p}_2$, 
$\sigma(\mathfrak{N}^-_1) = \mathfrak{N}^-_2$, and 
$\sigma(\mathfrak{N}^+_1) = \mathfrak{N}^+_2$.

The number of Shimura curves for which we were able to compute a fundamental domain of $\mathcal{T}_{\mathfrak{p}}$ for
$\Gamma^{\mathfrak{p}}_{\mathfrak{N}^-, \mathfrak{N}^+}$
is summarized in Table~\ref{table:shimura_counts}.

\begin{table}[ht!]
\centering
\caption{
Number of Shimura curves found that admit a $\mathfrak{p}$-adic uniformization 
by an unramified prime of inertia degree~1, grouped by number field degree and genus.
}
\label{table:shimura_counts}
\begin{tabular}{ccccc}
\toprule
\textbf{Number field degree} & \multicolumn{4}{c}{\textbf{Genus}} \\
\cmidrule(lr){2-5}
 & 0 & 1 & 2 & 3 \\
\midrule
2 & 18 & 41 & 34 & 46 \\
3 & 7  & 37 & 11 & 37 \\
4 & 29 & 50 & 61 & 53$^{1}$ \\
5 & 0  & 0  & 2  & 2  \\
6 & 2  & 7  & 12 & 8  \\
7 & 0  & 0  & 0  & 0  \\
\midrule
\textbf{Total} & \textbf{56} & \textbf{135} & \textbf{120} & \textbf{146} \\
\bottomrule
\end{tabular}
\vspace{0.5em}

\footnotesize{$^{1}$ Count might not be complete because of an implementation bug.}
\end{table}

The complete tables are available at \url{https://eloitor.github.io/btquotients/}.

Following the approach of Voight~\cite{voight-shimura-curves},
for a given number field, we can compute a finite list of Shimura curves that is complete up to a given genus by using the Selberg–Zograf bound.
We normalize the measure for the hyperbolic area as
\[
  \mu(D) = \frac{1}{2\pi} \iint_D \frac{dx\,dy}{y^2},
\]
so that an ideal triangle has area $1/2$.
The Selberg--Zograf bound~\cite[Lemma~1.1]{voight-shimura-curves} gives an upper bound on the area of a Shimura curve in terms of its genus:
\[
  A < \frac{64}{3}(g + 1).
\]

The area of the Shimura curve 
$X_{\mathfrak{p}\mathfrak{N}^-, \mathfrak{N}^+}$ 
is given by Shimizu’s formula~\cite[Eq.~1]{voight-shimura-curves}:
\[
  A = \frac{4}{(2\pi)^{2n}}\, d_F^{3/2}\, \zeta_F(2)\,
      \Phi(\mathfrak{p}\mathfrak{N}^-)\, \Psi(\mathfrak{N}^+),
\]
where $\zeta_F$ is the Dedekind zeta function, and
\[
  \Phi(\mathfrak{N})
  = N(\mathfrak{N})
  \prod_{\mathfrak{p} \mid \mathfrak{N}} \left(1 - \frac{1}{N(\mathfrak{p})}\right),
  \qquad
  \Psi(\mathfrak{N})
  = N(\mathfrak{N})
  \prod_{\mathfrak{p} \mid \mathfrak{N}} \left(1 + \frac{1}{N(\mathfrak{p})}\right).
\]

Using the bound 
$\zeta_F(2)\,\Phi(\mathfrak{p})\,\Phi(\mathfrak{N}^-)\,\Psi(\mathfrak{N}^+) \ge 1$
in the inequality
\begin{equation}\label{bound_discriminant}
  \frac{4}{(2\pi)^{2n}}\, d_F^{3/2}\,
  \zeta_F(2)\, \Phi(\mathfrak{p})\Phi(\mathfrak{N}^-)\Psi(\mathfrak{N}^+)
  < \frac{64}{3}(g + 1),
\end{equation}
we see that it suffices to consider fields whose discriminant satisfies
\[
  \frac{4}{(2\pi)^{2n}}\, d_F^{3/2} < \frac{64}{3}(g + 1).
\]

For each field $F$ of degree $n$ satisfying this bound, define the constant
\[
  C_F = \frac{3}{16}\,
  \frac{d_F^{3/2}\,\zeta_F(2)}{(2\pi)^{2n}},
\]
so that inequality~\eqref{bound_discriminant} becomes
\[
  C_F\,\Phi(\mathfrak{p})\,\Phi(\mathfrak{N}^-)\,\Psi(\mathfrak{N}^+) < g + 1.
\]

We list candidates for $\mathfrak{p}$, $\mathfrak{N}^-$, and $\mathfrak{N}^+$ as follows.
Using the elementary bounds
$N(\mathfrak{p}) - 1 \le \Phi(\mathfrak{p}\mathfrak{N}^-)$
and
$\Psi(\mathfrak{N}^+) \le N(\mathfrak{N}^+)$,
we deduce that it suffices to consider ideals $\mathfrak{N}^+$ of norm bounded by
\[
  \frac{g + 1}{C_F \bigl(N(\mathfrak{p}) - 1\bigr)}
\]
satisfying
\[
  \Psi(\mathfrak{N}^+) <
  \frac{g + 1}{C_F \bigl(N(\mathfrak{p}) - 1\bigr)}.
\]

Let $\mathfrak{q}_0$ be the smallest prime that splits in $F$.
For each $\mathfrak{N}^+$ of norm bounded by
\[
  \frac{g + 1}{C_F \bigl(N(\mathfrak{q}_0) - 1\bigr)},
\]
we consider all primes
$\mathfrak{p} \nmid \mathfrak{N}^+$,
unramified and of inertia degree~$1$,
satisfying
\[
  \Phi(\mathfrak{p})
  = N(\mathfrak{p}) - 1
  \le
  \frac{g + 1}{C_F \Phi(\mathfrak{N}^-)\Psi(\mathfrak{N}^+)}
  \le
  \frac{g + 1}{C_F \Psi(\mathfrak{N}^+)}.
\]

For each candidate pair $(\mathfrak{N}^+, \mathfrak{p})$, we list all possible ideals $\mathfrak{N}^-$ as follows.
Since we require $\mathfrak{N}^-$ to be square-free,
\[
  \Phi(\mathfrak{N}^-) = \prod_{\mathfrak{p} \mid \mathfrak{N}^-} \bigl(N(\mathfrak{p}) - 1\bigr).
\]
From this and the bound
\begin{equation}\label{bound_N_minus}
  \Phi(\mathfrak{N}^-) \le
  \frac{g + 1}{C_F \Phi(\mathfrak{p}) \Psi(\mathfrak{N}^+)},
\end{equation}
we observe that the prime factors of $\mathfrak{N}^-$ are bounded by
\[
  \frac{g + 1}{C_F \Phi(\mathfrak{p}) \Psi(\mathfrak{N}^+)} + 1.
\]
Finally, we construct $\mathfrak{N}^-$ as a product of distinct primes satisfying this bound,
not dividing $\mathfrak{p}\mathfrak{N}^+$,
such that the number of factors of $\mathfrak{N}^-$ has the same parity as the degree of $F$, and that inequality~\eqref{bound_N_minus} holds.

\bibliographystyle{amsalpha}
\bibliography{refs}
\end{document}